\pdfoutput=1
\documentclass[12pt,a4paper]{amsart}
\usepackage[utf8]{inputenc}
\usepackage[T1]{fontenc}
\usepackage[british]{babel}
\usepackage[babel=true,tracking=true]{microtype}
\usepackage[margin=3cm]{geometry}
\usepackage{amsmath,amsthm,amssymb,bookmark}
\usepackage[backend=biber,style=alphabetic,giveninits]{biblatex}
\DeclareNameAlias{default}{family-given}
\AtBeginBibliography{\small}

\theoremstyle{plain}
\newtheorem{theorem}{Theorem}[section]
\newtheorem{lemma}[theorem]{Lemma}
\newtheorem{corollary}[theorem]{Corollary}
\newtheorem{proposition}[theorem]{Proposition}

\theoremstyle{definition}
\newtheorem{definition}[theorem]{Definition}

\newtheorem{example}[theorem]{Example}
\newtheorem{question}[theorem]{Question}

\theoremstyle{remark}
\newtheorem{remark}[theorem]{Remark}
\newtheorem{notation}[theorem]{Notation}
\newtheorem*{conventions}{Notations and conventions}
\newtheorem*{acknowledgements}{Acknowledgements}
\newtheorem*{structure}{Structure of the paper}

\title[Belyi's theorem for smooth WCI of general type]{Belyi's theorem for smooth \\ complete intersections of general type in \\ generalised Grassmannians and \\ weighted projective spaces}
\author{Mikhail Ovcharenko}
\address
{
  \textnormal{Steklov Mathematical Institute of RAS, 8 Gubkina street, Moscow 119991, Russia.}
  \newline
  \textnormal{HSE University, Laboratory of Mirror Symmetry, 6 Usacheva str., Moscow 119048, Russia.}
}
\email{ovcharenko@mi-ras.ru}

\DeclareMathOperator{\Aut}{Aut}
\DeclareMathOperator{\Bl}{Bl}
\DeclareMathOperator{\Bs}{Bs}
\DeclareMathOperator{\Cl}{Cl}
\DeclareMathOperator{\Crit}{Crit}
\DeclareMathOperator{\codim}{codim}
\DeclareMathOperator{\height}{ht}
\DeclareMathOperator{\Gal}{Gal}

\DeclareMathOperator{\Hom}{Hom}
\DeclareMathOperator{\Pic}{Pic}
\DeclareMathOperator{\Proj}{Proj}
\DeclareMathOperator{\res}{res}
\DeclareMathOperator{\Sing}{Sing}
\DeclareMathOperator{\Spec}{Spec}
\DeclareMathOperator{\Stab}{Stab}
\DeclareMathOperator{\WDiv}{WDiv}

\begin{document}

\begin{abstract}
  We show that A.~Javanpeykar's proof of Belyi's theorem for smooth complete intersections of general type in ordinary projective spaces can be generalised to smooth complete intersections of general type in generalised Grassmannians and weighted projective spaces.

  We propose an approach to the generalisation of this result to smooth complete intersections of general type in more general Mori dream spaces.
\end{abstract}

\maketitle

\section{Introduction}

In this paper we obtain a Belyi-type characterisation of smooth complex complete intersections of general type in generalised Grassmannians and weighted projective spaces which can be defined over \(\overline{\mathbb{Q}}\). Throughout the paper all varieties are assumed to be defined over \(\mathbb{C}\). A smooth projective variety~\(X\) is of \emph{general type} if its Kodaira dimension is maximal, i.e., equals \(\dim(X)\).

Let us recall the classical Belyi's theorem.

\begin{theorem}[see~\cite{belyi/extensions,kock/belyi}]\label{theorem:belyi}
  A smooth complex projective curve \(X\) can be defined over \(\overline{\mathbb{Q}}\) if and only if there exists a surjective morphism \(f \colon X \rightarrow \mathbb{P}^1\) \'{e}tale over \(\mathbb{P}^1(\mathbb{C}) \setminus \{0, 1, \infty\}\).
\end{theorem}

Belyi's theorem is a starting point for the Grothendieck's theory of \emph{Dessins d'Enfants}: it implies that \(\Gal(\mathbb{Q})\) acts faithfully on the \'{e}tale fundamental group of \(\mathbb{P}^1(\mathbb{C}) \setminus \{0,1, \infty\}\) by outer automorphisms (see~\cite[Theorem~4.7.7]{szamuely/galois}). Moreover, \(\Gal(\mathbb{Q})\) acts faithfully on the set of connected components of the moduli space of surfaces of general type (see~\cites{bauer/belyi,easton/belyi}).

So it is natural to expect higher-dimensional analogues of Belyi's theorem. In the case \(\dim(X) \geqslant 2\) a morphism \(X \rightarrow \mathbb{P}^1\) is replaced by a \emph{Lefschetz function}.

\begin{definition}\label{definition:lefschetz}
  Let \(X\) be a smooth projective variety, and let \(\mathcal{H}\) be a very ample linear system on \(X\).

  A \emph{Lefschetz pencil} \(\mathcal{L} \subset \mathcal{H}\) is a one-dimensional linear subsystem such that \(\codim_X(\Bs(\mathcal{L})) > 1\), and any element \(X_t \in \mathcal{L}\) of the pencil has at worst a single ordinary double point.

  A \emph{Lefschetz function} (see~{\cites{gonzalez-diez/belyi,javanpeykar/belyi}}) is a composition of the rational map \(X \dashrightarrow \mathbb{P}^1\) defined by a Lefschetz pencil with a rational function on \(\mathbb{P}^1\).
\end{definition}

\begin{remark}
  In Theorem~\ref{theorem:belyi} we can always assume that all ramification indices of \(f\) at points lying above 1 are equal to 2 (see~\cite[Exercise~4.7]{szamuely/galois}), so Definition~\ref{definition:lefschetz} is consistent with the one-dimensional case.
\end{remark}

A Belyi-type theorem holds for smooth surfaces of general type by the result of G.~Gonz\'{a}lez-Diez (see~\cite[Theorem~1, Proposition~1]{gonzalez-diez/belyi}). Following the approach of Gonz\'{a}lez-Diez, A.~Javanpeykar has obtained a Belyi-type theorem for smooth complete intersections of general type in ordinary projective spaces.

\begin{theorem}[see~{\cite[Theorem~1.1]{javanpeykar/belyi}}]\label{theorem:javanpeykar}
  Let \(X \subset \mathbb{P}^n\) be a smooth complex complete intersection of general type of dimension at least 3. Then the variety \(X\) can be defined over \(\overline{\mathbb{Q}}\) if and only if there exists a Lefschetz function \(X \dashrightarrow \mathbb{P}^1\) with at most 3 critical points.
\end{theorem}

In this paper we show that the proof of Theorem~\ref{theorem:javanpeykar} can be modified for smooth complete intersections of general type in generalised Grassmannians and weighted projective spaces, using the existing results on the infinitesimal Torelli theorem for such varieties (see~\cites{konno/torelli,usui/torelli}). We refer the reader to Section~\ref{section:preliminaries} for preliminaries and notations for (weighted) complete intersections.

\begin{definition}
  A \emph{generalised Grassmannian} is a projective rational homogeneous variety with Picard number one.
\end{definition}

\begin{remark}
  Any generalised Grassmannian can be presented as \(G / P\), where \(G\) is a simple algebraic group, and \(P\) is a maximal parabolic subgroup (see~\cite[Section~2]{munoz/campana} for a survey of rational homogeneous spaces).
\end{remark}

Let us formulate the main results of the paper.

\begin{theorem}\label{theorem:CI-grassmannians}
  Let \(Y\) be a generalised Grassmannian, and \(X \subset Y\) be a smooth complex complete intersection of general type of dimension at least~3. Then the variety \(X\) can be defined over \(\overline{\mathbb{Q}}\) if and only if there exists a Lefschetz function \(X \dashrightarrow \mathbb{P}^1\) with at most 3 critical points.
\end{theorem}

\begin{theorem}\label{theorem:WCI}
  Let \(X \subset \mathbb{P}(\rho)\) be a smooth well-formed complex weighted complete intersection of general type of dimension at least 3. Assume that \(\dim(\vert \mathcal{O}_X(1) \vert) \geqslant 2\). Then the variety \(X\) can be defined over \(\overline{\mathbb{Q}}\) if and only if there exists a Lefschetz function \(X \dashrightarrow \mathbb{P}^1\) with at most 3 critical points.
\end{theorem}

\begin{remark}
  The only reason we assume \(\dim(\vert \mathcal{O}_X(1) \vert) \geqslant 2\) is because it is crucial for Usui's proof of the infinitesimal Torelli theorem (see~\cite{usui/torelli}). We are expecting that this restriction can be omitted (see Question~\ref{question:torelli}).
\end{remark}

\begin{structure}
  In Section~\ref{section:preliminaries} we briefly remind basic facts about (weighted) complete intersections and Cox rings. In Section~\ref{section:proof} we prove Theorems~\ref{theorem:CI-grassmannians} and~\ref{theorem:WCI}. In Section~\ref{section:generalisation} we discuss their possible generalisation to smooth well-formed complete intersection of general type in \emph{weighted generalised Grassmannians}. 
\end{structure}

\begin{conventions}
  For any variety \(Y\) we denote by \(\WDiv(Y)\) the group of its Weil divisors, and for any field \(\Bbbk\) we denote by \(Y(\Bbbk)\) the set of \(\Bbbk\)-points of \(Y\).

  Throughout the paper all varieties are assumed to be defined over \(\mathbb{C}\). However, most of our arguments work over any algebraically closed field of characteristic 0.
\end{conventions}

\begin{acknowledgements}
  The author is grateful to C.~Shramov for introducing us to the paper~\cite{javanpeykar/belyi} and useful discussions, to A.~Javanpeykar for encouraging us to write this paper and valuable comments, and to V.~Przyjalkowski for many helpful suggestions and careful reading of the paper. We also want to thank the referee for their useful remarks.
  
  This work was supported by the Russian Science Foundation under grant no.~24--71--10092, \url{https://rscf.ru/en/project/24-71-10092/}.
\end{acknowledgements}

\section{Preliminaries}\label{section:preliminaries}

\subsection{Complete intersections and Cox rings}\label{subsection:CI-generalities}

The goal of this subsection is to remind some generalities on complete intersections and Cox rings.

\subsubsection{Cox rings and Mori dream spaces}

\begin{definition}\label{definition:cox-ring}
  Let \(Y\) be an irreducible normal variety with \(\Cl(Y) \cong \mathbb{Z}^{\rho}\). Fix a subgroup \(K \subset \WDiv(Y)\) such that the canonical map \(K \rightarrow \Cl(Y)\) sending a Weil divisor \(D \in K\) to its class \([D] \in \Cl(Y)\) is an isomorphism. The \emph{Cox ring} of the variety \(Y\) is the following \(\Cl(Y)\)-graded algebra:
  \[
    \mathcal{R}(Y) = \bigoplus_{[D] \in \Cl(Y)} \mathcal{R}_{[D]}(Y), \quad
    \mathcal{R}_{[D]}(Y) = H^0(Y, \mathcal{O}_Y(D)),
  \]
  where the multiplication in \(\mathcal{R}(Y)\) is defined by multiplying homogeneous sections in the field of rational functions \(\mathbb{C}(Y)\).
\end{definition}

\begin{remark}
  The Cox ring \(\mathcal{R}(Y)\) is unique up to isomorphism, and \(\mathcal{R}(Y)\) is an integral domain (see~\cite[Construction~1.4.1.1, 1.5.1]{arzhantsev/cox}).
\end{remark}

\begin{definition}[see~{\cite[Definition~3.3.4.1]{arzhantsev/cox}}]
  An irreducible normal projective variety \(Y\) with divisor class group \(\Cl(Y) \cong \mathbb{Z}^{\rho}\) and finitely generated Cox ring \(\mathcal{R}(Y)\) is a \emph{Mori dream space} of rank \(\rho\).
\end{definition}

\begin{remark}
  Any Mori dream space has a natural decomposition of its cone of effective divisors into convex polyhedral sets called \emph{Mori chambers}. It encodes the birational geometry of a Mori dream space (see~\cite[\nopp 3.3.4]{arzhantsev/cox}).
\end{remark}

\begin{example}
  Any smooth Fano variety is a Mori dream space (see~\cite[Theorem~4.3.3.7]{arzhantsev/cox}). The Cox ring of a generalised flag variety \(G / P\) can be described in terms of the representation theory of the group \(G\) (see~\cite[\nopp 3.2.3]{arzhantsev/cox}).
\end{example}

\subsubsection{Correspondence between closed subschemes and homogeneous ideals}

There exists a natural correspondence between closed subschemes of a smooth Mori dream space \(Y\) and homogeneous ideals in its Cox ring \(\mathcal{R}(Y)\). More precisely, ideal sheaves on \(Y\) correspond to so-called ``saturated'' homogeneous ideals in \(\mathcal{R}(Y)\).

\begin{notation}[see~{\cite[Definition~2.2]{pieropan/galois}}]
  Let \(Y\) be an irreducible normal variety with divisor class group \(\Cl(Y) \cong \mathbb{Z}^{\rho}\). For any homogeneous element \(f \in \mathcal{R}(Y)\) we denote by \(D_f\) the corresponding effective divisor on \(Y\).
  \begin{enumerate}
  \item For any homogeneous ideal \(I \subset \mathcal{R}(Y)\) we denote by \(\varphi_Y(I)\) the ideal sheaf associated with the following sum of \(\mathcal{O}_Y\)-subsheaves:
    \[
      \varphi_Y(I) = \sum_{\substack{f \in I \\ \text{homogeneous}}} \mathcal{O}_Y(-D_f) \subseteq \mathcal{O}_Y.
    \]
    In particular, \(\varphi_Y(I)\) corresponds to a unique closed subscheme in \(Y\) (for example, see~\cite[Proposition~II.5.9]{hartshorne/algebraic}).
  \item For any ideal sheaf \(\mathcal{I} \subseteq \mathcal{O}_Y\) let \(\psi_Y(\mathcal{I}) \subset \mathcal{R}(Y)\) be the ideal generated by all homogeneous elements \(f \in \mathcal{R}(Y)\) such that \(\mathcal{O}_Y(-D_f) \subseteq \mathcal{I}\).
  \end{enumerate}
\end{notation}

\begin{remark}\label{remark:MDS-one-closed-subschemes}
  If \(\Cl(Y) \simeq \mathbb{Z}\), then \(Y \simeq \Proj(\mathcal{R}(Y))\) (see~\cite[\nopp 3.3.4]{arzhantsev/cox}). For any homogeneous ideal \(I \subset \mathcal{R}(Y)\) the quotient map \(\mathcal{R}(Y) \twoheadrightarrow \mathcal{R}(Y) / I\) defines the closed subscheme \(\Proj(\mathcal{R}(Y) / I) \subset \Proj(\mathcal{R}(Y)) = Y\). Its ideal sheaf can be identified with the ideal sheaf \(\varphi_Y(I)\) introduced above (for example, see~\cite[Proposition~II.5.9]{hartshorne/algebraic}).
\end{remark}

\begin{proposition}\label{proposition:MDS-smooth-correspondence}
  Let \(Y\) be a smooth Mori dream space, \(\mathcal{R}(Y)\) be its Cox ring, and \(B \subset \mathcal{R}(Y)\) be the irrelevant ideal.
  \begin{enumerate}
  \item Let \(I = (f_1, \ldots, f_m) \subset \mathcal{R}(Y)\) be a homogeneous ideal. Then the following decomposition holds: \(\varphi_Y(I) = \sum_{j = 1}^m \mathcal{O}_Y(-D_{f_j})\).
  \item For any homogeneous ideal \(I \subset \mathcal{R}(Y)\) we have \(\psi_Y(\varphi_Y(I)) = (I : B^{\infty})\), where \((I : B^{\infty})\) is the saturation of the ideal \(I\) with respect to \(B\):
    \[
      (I : B^{\infty}) = \bigcup_{l = 1}^{\infty} \{f \in \mathcal{R}(Y) \mid f B^{l} \subset I\}.
    \]
  \item For any ideal sheaf \(\mathcal{I} \subseteq \mathcal{O}_Y\) we have \(\varphi_Y(\psi_Y(\mathcal{I})) = \mathcal{I}\).
  \end{enumerate}
\end{proposition}

\begin{proof}
  If \(Y\) is smooth, then \(\Cl(Y) = \Pic(Y)\), so the Cox ring \(\mathcal{R}(Y)\) is \(\Pic(Y)\)-graded. Now we can apply~\cite[Proposition~3.1(1,2,3)]{pieropan/galois}.
\end{proof}

\begin{remark}
  The irrelevant ideal \(B \subset \mathcal{R}(Y)\) can be explicitly described if one fixes an ample divisor \(D \in K\) on \(Y\) (see~\cite[Corollary~1.6.3.6]{arzhantsev/cox}), where the subgroup \(K \subset \WDiv(Y)\) was introduced in Definition~\ref{definition:cox-ring}.
\end{remark}

We are going to describe a similar correspondence for singular Mori dream spaces~\(Y\) with \(\Cl(Y) \simeq \mathbb{Z}\). Non-triviality of \(\Cl(Y) / \Pic(Y)\) imposes a notion of ``saturation'' for homogeneous ideals \(I \subset \mathcal{R}(Y)\) different from \((I : B^{\infty})\).

\begin{notation}\label{notation:saturation}
  Let \(Y\) be a Mori dream space with \(\Cl(Y) \simeq \mathbb{Z}\). For any homogeneous ideal \(I \subset \mathcal{R}(Y)\) we use the following notation:
  \[
    I^{\infty} = (\{x \in \mathcal{R}(Y) \mid(\mathcal{R}(Y) \cdot x)_{m \cdot k} \subset I
    \text{ for all } k \gg 0 \}), \; m = [\Cl(Y) : \Pic(Y)].
  \]
\end{notation}

\begin{remark}\label{remark:saturation}
  Note that if \(Y\) is smooth, then \(I^{\infty}\) equals to the usual saturation \((I : B^{\infty})\) with respect to the irrelevant ideal \(B \subset \mathcal{R}(Y)\). One can also check that any radical homogeneous ideal \(I \subset \mathcal{R}(Y)\) satisfies \(I = I^{\infty}\).
\end{remark}

\begin{proposition}\label{proposition:MDS-one-correspondence}
  Let \(Y\) be a Mori dream space with \(\Cl(Y) \simeq \mathbb{Z}\), let \(\mathcal{R}(Y)\) be its Cox ring, and let \(B \subset \mathcal{R}(Y)\) be the irrelevant ideal.
  \begin{enumerate}
  \item Let \(I = (f_1, \ldots, f_m) \subset \mathcal{R}(Y)\) be a homogeneous ideal. Then the following decomposition holds: \(\varphi_Y(I) = \sum_{j = 1}^m \mathcal{O}_Y(-D_{f_j})\).
  \item For any homogeneous ideal \(I \subset \mathcal{R}(Y)\) we have \(\psi_Y(\varphi_Y(I)) = I^{\infty}\).
  \item For any ideal sheaf \(\mathcal{I} \subseteq \mathcal{O}_Y\) we have \(\varphi_Y(\psi_Y(\mathcal{I})) = \mathcal{I}\).
  \end{enumerate}
\end{proposition}

\begin{proof}
  We want to derive the statement from~\cite[Proposition~3.1]{pieropan/galois}, as was done in the proof of Proposition~\ref{proposition:MDS-smooth-correspondence}. Yet, we cannot apply it directly, because in our case the Cox ring \(\mathcal{R}(Y)\) is \(\Cl(Y)\)-graded, not \(\Pic(Y)\)-graded.

  We will need the following considerations.

  Firstly, let us recall that \(Y \simeq \Proj(\mathcal{R}(Y))\) (see~\cite[\nopp 3.3.4]{arzhantsev/cox}). We introduce the \emph{truncated} Cox ring, which is \(\Pic(Y)\)-graded:
  \[
    \mathcal{R}(Y)^{[m]} = \bigoplus_{k = 0}^{\infty} H^0(Y, \mathcal{O}_Y(m \cdot k)), \quad
    m = [\Cl(Y) : \Pic(Y)].
  \]
  Recall that \(\Proj(\mathcal{R}(Y))\) is canonically isomorphic to \(\Proj(\mathcal{R}(Y)^{[m]})\) (see~\cite[2.4.7(i)]{grothendieck/elements}). Then any ideal sheaf on \(Y\) can be canonically identified with an ideal sheaf on \(\widetilde{Y} := \Proj(\mathcal{R}(Y)^{[m]})\), and vice versa. In particular, for any ideal \(I \subset \mathcal{R}(Y)\) and its truncation \(\widetilde{I} = I \cap \mathcal{R}(Y)^{[m]}\) we can identify the ideal sheaves \(\phi_Y(I)\) and \(\phi_{\widetilde{Y}}(\widetilde{I})\).

  Secondly, let \(I \subset \mathcal{R}(Y)\) be an homogeneous ideal, and \(\widetilde{I} = I \cap \mathcal{R}(Y)^{[m]}\) be its image in the truncated Cox ring. It is not hard to check that \(I\) is saturated in the sense of Notation~\ref{notation:saturation} (i.e., we have \(I = I^{\infty}\)) if and only if \(\widetilde{I}\) is saturated in the usual sense with respect to the irrelevant ideal \(\widetilde{B} \subset \mathcal{R}(Y)^{[m]}\) (i.e., we have \(\widetilde{I} = (\widetilde{I} : \widetilde{B}^{\infty})\)) (for example, see the proof of~\cite[Lemma~2.10]{ovcharenko/classification}).

  Now we can prove the statement.
  \begin{enumerate}
  \item[(1)] Let us consider the truncation \(\widetilde{I} = I \cap \mathcal{R}(Y)^{[m]}\) of the ideal \(I \subset \mathcal{R}(Y)\). If \(\widetilde{I}\) is generated by \(g_1, \ldots, g_l \in \mathcal{R}(Y)^{[m]}\), then we have
\(\varphi_{\widetilde{Y}}(\widetilde{I}) = \sum_{j = 1}^l \mathcal{O}_{\widetilde{Y}}(-D_{g_j})\) by~\cite[Proposition~3.1(1)]{pieropan/galois}, where \(\widetilde{Y} = \Proj(\mathcal{R}(Y)^{[m]})\). Then the canonical isomorphism \(Y \simeq \widetilde{Y}\) implies that we have \(\varphi_Y(I) = \sum_{j = 1}^m \mathcal{O}_Y(-D_{f_j})\) as well.
\item[(3)] For any ideal sheaf \(\mathcal{I} \subseteq \mathcal{O}_Y\) we can identify the ideal \(\psi_{\widetilde{Y}}(\mathcal{I}) \subset \mathcal{R}(Y)^{[m]}\) with the truncation \(\widetilde{\psi_Y(\mathcal{I})} = \psi_Y(\mathcal{I}) \cap \mathcal{R}(Y)^{[m]}\) of the ideal \(\psi_Y(\mathcal{I})\). Then we can apply~\cite[Proposition~3.1(3)]{pieropan/galois}.

  \item[(2)] The ideal \(I \subset \mathcal{R}(Y)\) defines an ideal sheaf \(\varphi_Y(I)\) on \(Y\). After the canonical isomorphism \(Y \simeq \widetilde{Y}\) we can think of \(\varphi_Y(I)\) as an ideal sheaf on \(\widetilde{Y}\). According to~\cite[Proposition~3.1(2)]{pieropan/galois}, ideal sheaves on \(\widetilde{Y}\) are in bijection with the saturated ideals \(\widetilde{I} = (\widetilde{I} : \widetilde{B}^{\infty})\) in \(\mathcal{R}(Y)^{[m]}\). But we can identify \(\widetilde{I}\) with the truncation of the ideal \(\psi_Y(\varphi_Y(I))\), so we obtain that \(\psi_Y(\varphi_Y(I)) = I^{\infty}\). \qedhere
  \end{enumerate} 
\end{proof}

Propositions~\ref{proposition:MDS-smooth-correspondence} and~\ref{proposition:MDS-one-correspondence} motivate the following definitions.

\begin{definition}\label{definition:saturation}
  Let \(Y\) a Mori dream space, and \(I \subset \mathcal{R}(Y)\) be a homogeneous ideal. The \emph{saturation} of \(I\) is the ideal \((I : B^{\infty})\) if \(Y\) is smooth, and the ideal~\(I^{\infty}\) if \(\Cl(Y) \simeq \mathbb{Z}\) (see Remark~\ref{remark:saturation}). An ideal is \emph{saturated} if it equals its saturation.
\end{definition}

\begin{definition}
  Let \(X \subseteq Y\) be a closed subscheme in a Mori dream space such that \(Y\) is smooth, or \(\Cl(Y) \simeq \mathbb{Z}\). We refer to \(\psi_Y(\mathcal{I}_X) \subset \mathcal{R}(Y)\) as the \emph{defining ideal} of \(X\). We define the \emph{homogeneous coordinate ring} of \(X\) as the \(\Cl(Y)\)-graded ring \(\mathcal{R}_Y(X) = \mathcal{R}(Y) / \psi_Y(\mathcal{I}_X) = \oplus_{i \in \Cl(Y)} \mathcal{R}_Y(X)_i\).
\end{definition}

\subsubsection{Complete intersections in Mori dream spaces}

\begin{definition}
  Let \(Y\) be an irreducible normal variety. A \emph{complete intersection} of hypersurfaces \(H_1, \ldots, H_c \subset Y\) is a closed subscheme \(X \subseteq Y\) with \(\codim_Y(X) = c\) and the ideal sheaf equal to \(\mathcal{I}_X = \mathcal{I}_{H_{1}} + \cdots +\mathcal{I}_{H_c}\).
\end{definition}

\begin{definition}[see~{\cite[Definition~2.2]{pieropan/galois}}]
  Let \(Y\) be an irreducible normal variety with \(\Cl(Y) \cong \mathbb{Z}^{\rho}\). A complete intersection \(X \subseteq Y\) of hypersurfaces \(H_{1}, \ldots, H_{c}\) is \emph{strict} if its ideal \(\psi_Y(\mathcal{I}_X) \subset \mathcal{R}(Y)\) is generated by \(c\) elements.
\end{definition}

\begin{remark}\label{remark:CI-MDS-one}
  Let \(X \subset Y\) be a strict complete intersection in a Mori dream space with \(\Cl(Y) \simeq \mathbb{Z}\). The ideal \(\psi_Y(\mathcal{I}_X)\) is generated by \(\codim_Y(X)\) elements. Remark~\ref{remark:MDS-one-closed-subschemes} implies that we can identify \(X\) with \(\Proj(\mathcal{R}(Y) / \psi_Y(\mathcal{I}_X))\). Then we have
  \[
    \codim_Y(X) = \dim(Y) - \dim(X) =
    \dim(\mathcal{R}(Y)) - \dim(\mathcal{R}(Y) / \psi_Y(\mathcal{I}_X)) = \height(\psi_Y(\mathcal{I}_X)).
  \]
  Then \(\psi_Y(\mathcal{I}_X)\) is generated by a regular sequence by~\cite[Proposition~1.5.11]{bruns/cohenmacaulay}.
\end{remark}

\begin{lemma}\label{lemma:CI-MDS-strictness}
  Let \(Y\) be a Mori dream space, and \(X \subseteq Y\) be a complete intersection of hypersurfaces \(D_{f_1}, \ldots, D_{f_c}\) defined by homogeneous elements \(f_j \in \mathcal{R}(Y)\), and let \(B = (g_1, \ldots, g_s) \subset \mathcal{R}(Y)\) be the irrelevant ideal.

  Assume that one of the following conditions hold:
  \begin{itemize}
  \item \(Y\) is smooth, and \(s < \height(B)\);
  \item \(\Cl(Y) \simeq \mathbb{Z}\), and the ideal \((f_1, \ldots, f_c)\) is radical.
  \end{itemize}
  Then we have \(\psi_Y(\mathcal{I}_X) = (f_1, \ldots, f_c)\), and \(\height(\psi_Y(\mathcal{I}_X)) = c\). In particular, \(X\) is a strict complete intersection, and \((f_1, \ldots, f_c)\) is a regular sequence.  
\end{lemma}

\begin{proof}
  Put \(I = (f_1, \ldots, f_c) \subset \mathcal{R}(Y)\).

  Assume that \(Y\) is smooth, and \(s < \height(B)\). By Proposition~\ref{proposition:MDS-smooth-correspondence}(1) we can identify \(\varphi_Y(I)\) with the ideal sheaf \(\mathcal{I}_X\). Accordingly to~\cite[Lemma~4.2]{pieropan/galois}, we have \(\height(I) = c\). Then by~\cite[Lemma~4.1]{pieropan/galois} it is saturated with respect to the irrelevant ideal \(B\). Proposition~\ref{proposition:MDS-smooth-correspondence}(2) implies that \(\psi_Y(\varphi_Y(I)) = I\). Consequently, we have \(\psi_Y(\mathcal{I}_X) = I\).

  Assume that \(\Cl(Y) \simeq \mathbb{Z}\), and \(I\) is a radical ideal. By Proposition~\ref{proposition:MDS-one-correspondence}(1) we can identify \(\varphi_Y(I)\) with the ideal sheaf \(\mathcal{I}_X\). It is not hard to check that a radical homogeneous ideal \(I \subset \mathcal{R}(Y)\) satisfies \(I = I^{\infty}\). Proposition~\ref{proposition:MDS-one-correspondence}(2) implies that \(\psi_Y(\varphi_Y(I)) = I\). Consequently, we have \(\psi_Y(\mathcal{I}_X) = I\). Remark~\ref{remark:MDS-one-closed-subschemes} implies that \(\Proj(\mathcal{R}(Y) / I) \simeq X\). From \(Y \simeq \Proj(\mathcal{R}(Y))\) (see~\cite[\nopp 3.3.4]{arzhantsev/cox}) we obtain that \(\height(I) = \dim(\mathcal{R}(Y)) - \dim(\mathcal{R}(Y) / I) = \codim_Y(X) = c\) (see Remark~\ref{remark:CI-MDS-one}).

  In both cases \(\psi_Y(\mathcal{I}_X) = I\) is generated by \(c\) elements, and \(\height(\psi_Y(\mathcal{I}_X)) = c\). Then \((f_1, \ldots, f_c)\) is a regular sequence by~\cite[Proposition~1.5.11]{bruns/cohenmacaulay}.
\end{proof}

\begin{remark}
  The assumptions of Lemma~\ref{lemma:CI-MDS-strictness} cannot be omitted.

  Consider a Mori dream space \(Y = \mathbb{P}^{1} \times \mathbb{P}^{1}\) with coordinates \((x_0 : x_1)\), \((y_0 : y_1)\), and the closed subscheme \(X = ((0 : 1), (0 : 1))\). Its defining ideal has the form  \(\psi_Y(\mathcal{I}_X) = (x_0, y_0) \subset \mathcal{R}(Y) = \mathbb{C}[x_0, x_1, y_0, y_1]\). But we can also define \(X\) by a non-saturated ideal \(I = (x_0, x_1 y_0)\) (see~\cite[Example~4.6]{pieropan/galois}).

  Consider a Mori dream space \(Y = \mathbb{P}(1, 1, 3, 6)\) with coordinates \((x_0, \ldots, x_3)\) and the homogeneous ideal \(I = (x_0^3, x_1^3) \subset \mathcal{R}(Y) = \mathbb{C}[x_0, x_1, x_2, x_3]\). Its saturation \(\psi_Y(\phi_Y(I)) = I^{\infty} = (x_0^3, x_1^3, x_0^2 x_1^2)\) cannot be generated by a regular sequence, hence \(I\) defines a non-strict complete intersection in \(Y\) (see~\cite{przyjalkowski/weighted}).
\end{remark}

\subsubsection{Application: conjugate complete intersections}

Let \(Y\) be a Mori dream space with \(\Cl(Y) \simeq \mathbb{Z}\), and let \(X \subseteq Y\) be a strict complete intersection. Then Remark~\ref{remark:CI-MDS-one} implies that \(\psi_Y(\mathcal{I}_X)\) is generated by a regular sequence of homogeneous elements.

\begin{definition}
  Let \(X \subseteq Y\) be a strict complete intersection in a Mori dream space with \(\Cl(Y) \simeq \mathbb{Z}\), so \(\psi_Y(\mathcal{I}_X) \subset \mathcal{R}(Y)\) is generated by a regular sequence \((f_1, \ldots, f_c)\) of homogeneous elements. Put \(d_j = \deg(f_j) \in \mathbb{Z}_{> 0}\).

  We refer to \((d_1, \ldots, d_c)\) as the \emph{multidegree} of the complete intersection \(X \subseteq Y\).
\end{definition}

\begin{definition}
  Let \(X = \Proj(\mathbb{C}[x_0, \ldots, x_d] / I_X)\) be a projective variety, and \(\sigma \in \Aut(\mathbb{C} / \mathbb{Q})\). The \emph{conjugate variety} of \(X\) is \(X^{\sigma} = \Proj(\mathbb{C}[x_0, \ldots, x_d] / I_X^{\sigma})\), where \(\Aut(\mathbb{C} / \mathbb{Q})\) acts on coefficients of polynomials in \(\mathbb{C}[x_0, \ldots, x_d]\).
\end{definition}

\begin{lemma}\label{lemma:CI-MDS-one-conjugation}
  Let \(Y\) be a complex Mori dream space with \(\Cl(Y) \simeq \mathbb{Z}\) which can be defined over~\(\mathbb{Q}\), and let \(X \subseteq Y\) be a strict complete intersection of multidegree \((d_1, \ldots, d_c)\). Then its conjugate \(X^{\sigma}\) is also a strict complete intersection in \(Y\) of the same multidegree \(\mu\) for any \(\sigma \in \Aut(\mathbb{C} / \mathbb{Q})\).
\end{lemma}

\begin{proof}
  By assumption \(X\) is a strict complete intersection, so its ideal \(\psi_Y(\mathcal{I}_X) \subset \mathcal{R}(Y)\) is generated by a regular sequence of homogeneous elements \((f_1, \ldots, f_c)\), where \(\deg(f_j) = d_j\) (see Remark~\ref{remark:CI-MDS-one}). Note that the action of \(\sigma \in \Aut(\mathbb{C} / \mathbb{Q})\) maps the ideal \(\psi_Y(\mathcal{I}_X) \subset \mathcal{R}(Y)\) to \(\psi_Y(\mathcal{I}_X)^{\sigma} = (f_1^{\sigma}, \ldots, f_c^{\sigma}) \subset \mathcal{R}(Y^{\sigma}) \simeq \mathcal{R}(Y)\).

  This action preserves the saturatedness (see Definition~\ref{definition:saturation}), because it preserves the irrelevant ideal and homogeneous components of \(\mathcal{R}(Y)\), and acts \(\mathbb{Q}\)-linearly on coefficients of any polynomial. Consequently, Proposition~\ref{proposition:MDS-one-correspondence} implies that \(\psi_Y(\mathcal{I}_X)^{\sigma} = \psi_Y(\mathcal{I}_{X^{\sigma}})\). It is not hard to check that the \(\mathbb{Q}\)-linearity of this action also preserves the property of the sequence \((f_1^{\sigma}, \ldots, f_c^{\sigma})\) to be regular. 
\end{proof}

\subsection{Weighted complete intersections}\label{subsection:WCI}

Here we remind the basic properties of weighted complete intersections. We refer the reader to Subsection~\ref{subsection:CI-generalities} for the terminology and notations related to Cox rings.

\begin{notation}
  Let \(\rho = (a_0, \ldots, a_N)\) be a tuple of positive integers. We also put \(R^{\rho} = \mathbb{C}[X_0, \ldots, X_N]\), where the grading \(R^{\rho} = \oplus_{n = 0}^{\infty} R_n^{\rho}\) is defined by \(\deg(X_i) = a_i\).
\end{notation}

\begin{definition}
  Let \(\rho = (a_0, \ldots, a_N)\) be a tuple of positive integers. We refer to \(\mathbb{P}(\rho) = \Proj(R^{\rho})\) as the \emph{weighted projective space with weights \(\rho\)}.
\end{definition}

\begin{definition}[{\cite[Definition~5.11]{ianofletcher/weighted}}]
  A weighted projective space \(\mathbb{P}(\rho)\), where \(\rho = (a_0, \ldots, a_N)\), is said to be \emph{well-formed} if
  \[
    \gcd(a_0, \ldots, a_{i-1}, \widehat{a_i}, a_{i + 1}, \ldots, a_N) = 1, \quad i = 0, \ldots, N.
  \]
\end{definition}

\begin{proposition}[{\cite[\nopp 1.3.1]{dolgachev/weighted}}]
  Any weighted projective space is isomorphic to a well-formed one.
\end{proposition}

\begin{lemma}[{\cite[\nopp 1.4.1]{dolgachev/weighted}}]
  Let \(\mathbb{P}(\rho)\) be a well-formed weighted projective space. Then we have \(\mathcal{R}(\mathbb{P}(\rho)) \simeq R^{\rho}\).
\end{lemma}

\begin{definition}
  Let \(\mathbb{P}(\rho)\) be a weighted projective space. A closed subscheme \(X \subseteq \mathbb{P}(\rho)\) is a \emph{weighted complete intersection} of multidegree \(\mu = (d_1, \ldots, d_c)\) if its ideal \(\psi_Y(\mathcal{I}_X) \subset \mathcal{R}(\mathbb{P}(\rho))\) is generated by a regular sequence \((f_1, \ldots, f_c)\) of weighted homogeneous polynomials of degrees \(\deg(f_j) = d_j \in \mathbb{Z}_{> 0}\).
\end{definition}

\begin{definition}[see~{\cite[Definition~2.1, Proposition~2.3]{sano/hypersurfaces}}]
  Let \(X \subseteq \mathbb{P}(\rho)\) be a weighted complete intersection, and \(\mathcal{R}_Y(X) = \mathcal{R}(Y) / \psi_Y(\mathcal{I}_X)\) be its homogeneous coordinate ring. We define the sheaf \(\mathcal{O}_X(k)\) as the coherent sheaf on \(X\) associated to the graded \(\mathcal{R}_Y(X)\)-module \(R(k)\) whose degree \(i\) part is \(R(k)_i = \mathcal{R}_Y(X)_{k + i}\).
\end{definition}

\begin{definition}[{\cite[Definition~6.3]{ianofletcher/weighted}}]
  A closed subscheme \(X \subseteq \mathbb{P}(\rho)\) is said to be \emph{quasi-smooth} if its affine cone \(\Spec(R^{\rho} / \psi_Y(\mathcal{I}_X))\) is smooth outside the origin, where \(\psi_Y(\mathcal{I}_X) \subset R^{\rho}\) is the defining ideal.
\end{definition}

\begin{definition}[cf.~{\cite[Definition~1.1]{dimca/singularities}}]\label{definition:WF-subscheme}
  A closed subscheme \(X \subseteq \mathbb{P}(\rho)\) is said to be \emph{well-formed} if \(\mathbb{P}(\rho)\) is well-formed, and \(\codim_X(X \cap \Sing(\mathbb{P}(\rho))) \geqslant 2\).
\end{definition}

\begin{proposition}[see~{\cite[Proposition~8]{dimca/singularities}} and~{\cite[Corollary~2.14]{przyjalkowski/bounds}}]\label{proposition:WCI-smooth-WF-criterion}
  Let \(X \subseteq \mathbb{P}(\rho)\) be a weighted complete intersection. The following assertions are equivalent:
  \begin{itemize}
  \item \(X\) is smooth and well-formed;
  \item \(\mathbb{P}(\rho)\) is well-formed, \(X\) is quasi-smooth, and \(X \cap \Sing(\mathbb{P}(\rho)) = \varnothing\).
  \end{itemize}
\end{proposition}

\begin{example}
  If a weighted complete intersection is not well-formed, various pathologies can arise.

  \begin{itemize}
  \item There exists a K3 surface which can be realised as a smooth and quasi-smooth hypersurface \(X\) of degree 9 in \(\mathbb{P}(1, 2, 2, 3)\) which is not well-formed (see~\cite[\nopp 6.15(ii)]{ianofletcher/weighted}). The naive adjunction formula (see Proposition~\ref{proposition:WCI-QS-WF-adjunction}) would imply that \(\omega_X \simeq \mathcal{O}_X(1)\), which is nonsense.
  \item A general weighted hypersurface of multidegree 6 in \(\mathbb{P}(2, 3, 5^{(t)})\) is not well-formed or quasi-smooth for any \(t > 0\); nonetheless, it is smooth (see~\cite[Example~2.9]{przyjalkowski/on-automorphisms}).
  \end{itemize}
\end{example}

\begin{lemma}[{\cite[Corollary~3.3]{przyjalkowski/automorphisms}}]\label{lemma:WCI-QS-WF-projective-normality}
  Let \(X \subseteq \mathbb{P}(\rho)\) be a quasi-smooth well-formed weighted complete intersection. The restriction map is surjective for any \(m \in \mathbb{Z}_{\geqslant 0}\):
  \[
    \res \colon H^0 (\mathbb{P}(\rho), \mathcal{O}_{\mathbb{P}(\rho)}(m)) \rightarrow H^0(X, \mathcal{O}_X(m)).
  \]
\end{lemma}

\begin{definition}[{\cite[Definition~6.5]{ianofletcher/weighted}}]
  Let \(X \subseteq \mathbb{P}(a_0, \ldots, a_N)\) be a weighted complete intersection of multidegree \((d_1, \ldots, d_c)\). We refer to \(X\) as
  \begin{itemize}
  \item \emph{an intersection with a linear cone} if we have \(a_i = d_j\) for some \(i\) and \(j\);
  \item \emph{degenerate} if we have \(d_j = 1\) for some \(j = 1, \ldots, c\).
\end{itemize}
\end{definition}

\begin{lemma}[{\cite[Theorem~3.4.4]{dolgachev/weighted}}, {\cite[Lemma~7.1]{ianofletcher/weighted}}]\label{lemma:WCI-QS-WF-O1-dimension}
  Let \(X \subseteq \mathbb{P}(\rho)\) be a quasi-smooth well-formed weighted complete intersection which is non-degenerate. Put \(\rho = (a_0, \ldots, a_N)\). The following identities hold:
  \[
    \dim(H^0(X, \mathcal{O}_X(1))) = \dim(H^0(\mathbb{P}(\rho), \mathcal{O}_{\mathbb{P}(\rho)}(1))) = \vert \{i \mid a_i = 1\} \vert.
  \]
\end{lemma}

\begin{remark}
  Not being an intersection with a linear cone is a rather mild restriction, provided that \(X \subseteq \mathbb{P}(\rho)\) is sufficiently general (see~\cite[Proposition~2.9]{przyjalkowski/codimension}).
\end{remark}

\begin{proposition}[{\cite[Remark~4.2]{okada/rationality}}, {\cite[Proposition~2.3]{pizzato/nonvanishing}}]\label{proposition:WCI-QS-WF-lefschetz}
  Let \(X \subset \mathbb{P}(\rho)\) be a quasi-smooth well-formed weighted complete intersection of dimension at least~3. Then \(\Cl(X)\) is generated by the class of the divisorial sheaf \(\mathcal{O}_X(1)\).
\end{proposition}

\begin{proposition}[{\cite[Theorem~3.3.4]{dolgachev/weighted}}, {\cite[\nopp 6.14]{ianofletcher/weighted}}, {\cite[Corollary~2.6]{sano/hypersurfaces}}]\label{proposition:WCI-QS-WF-adjunction}
  Let \(X \subset \mathbb{P}(a_0, \ldots, a_N)\) be a quasi-smooth well-formed weighted complete intersection of dimension at least 2 and multidegree \((d_1, \ldots, d_c)\), and \(\omega_X\) be the dualising sheaf. Then the following identity holds: \(\omega_X \simeq \mathcal{O}_X(\sum_{j = 1}^c d_j - \sum_{i = 0}^N a_i)\).
\end{proposition}

\section{Proof of Theorems~\ref{theorem:CI-grassmannians} and~\ref{theorem:WCI}}\label{section:proof}

In this section we prove Theorems~\ref{theorem:CI-grassmannians} and~\ref{theorem:WCI}. To this end, we introduce the following class of varieties which behave like smooth complete intersections in \(\mathbb{P}^n\).

\begin{definition}\label{definition:CI-admissible}
  Let \(Y\) be a complex Mori dream space with \(\Cl(Y) \simeq \mathbb{Z}\) which can be defined over~\(\mathbb{Q}\) (see Subsection~\ref{subsection:CI-generalities}). A smooth complete intersection \(X \subseteq Y\) of multidegree \((d_1, \ldots, d_c)\) is \emph{admissible} if
\begin{itemize}
\item \(\Pic(X) \simeq \mathbb{Z}\), and is generated by the restriction \(\left.\mathcal{H}\right|_X\) of the ample generator \(\mathcal{H} \in \Cl(Y)\), provided that \(\dim(X) > 2\) (``Lefschetz theorem'');
  \item the adjunction formula \(\omega_X \simeq \left.\omega_Y\right|_X \otimes \left.\mathcal{H}\right|_X^{\otimes \sum d_j}\) holds if \(\dim(X) > 1\);
  \item the restriction map \(H^0(Y, \mathcal{H}^{\otimes k}) \rightarrow H^0(X, \left.\mathcal{H}\right|_X^{\otimes k})\) is surjective for any \(k \in \mathbb{Z}_{\geqslant 0}\), and for any divisor \(D \in H^0(X, \left.\mathcal{H}\right|_X^{\otimes k})\) there should exist \(\widetilde{X} \in H^0(Y, \mathcal{H}^{\otimes k})\) such that \(X = \widetilde{X} \cap Y\) is an admissible complete intersection in \(Y\);
  \item for any admissible complete intersection \(X \subset Y\) and any \(\sigma \in \Aut(\mathbb{C} / \mathbb{Q})\) its conjugate \(X^{\sigma}\) is also an admissible complete intersection (cf. Lemma~\ref{lemma:CI-MDS-one-conjugation});
  \item any smooth divisor \(D \in H^0(X, \left.\mathcal{H}\right|_X^{\otimes k})\) of general type in a very ample linear system satisfies the \emph{infinitesimal Torelli theorem}, i.e., the canonical map
    \[
      H^1(X, \mathcal{T}_X) \rightarrow \bigoplus_{p + q = \dim(X)} \Hom(H^p(X, \Omega_X^q), H^{p + 1}(X, \Omega_X^{q - 1}))
    \]
    is injective, where \(\mathcal{T}_X\) is the tangent sheaf of \(X\).
  \end{itemize}
\end{definition}

\begin{remark}
  There exist smooth varieties with very ample canonical class such that infinitesimal Torelli theorem fails for them (see~\cite{garra/torelli}).
\end{remark}

\begin{proposition}\label{proposition:CI-admissible}
  The following smooth complete intersections are admissible:
  \begin{itemize}
  \item complete intersections in generalised Grassmannians;
  \item well-formed weighted complete intersections with \(\dim(\vert \mathcal{O}_X(1) \vert) \geqslant 2\).
  \end{itemize}  
\end{proposition}

\begin{proof}
  A generalised Grassmannian \(Y\) is a quotient of a simple algebraic group~\(G\) by a maximal parabolic subgroup \(P\). It is a smooth Fano variety of Picard rank one (see~\cite[Section~2]{munoz/campana}). Moreover, any parabolic subgroup \(P \subset Y\) is conjugated to a standard parabolic subgroup (for example, see~\cite[Proposition~12.2]{malle/groups}), hence the variety \(Y\) can be defined over \(\mathbb{Q}\). Note that Lefschetz hyperplane section theorem clearly holds for \(X \subseteq Y\). Moreover, we have the adjunction formula for any smooth complete intersection \(X \subseteq Y\). Actually, its normal bundle \((\mathcal{I}_X / \mathcal{I}_X^2)^{\vee}\) is always isomorphic to \(\oplus_{j = 1}^c \mathcal{O}_X(d_j)\), where \((d_1, \ldots, d_c)\) is the multidegree of \(X\), because the ideal sheaf \(\mathcal{I}_X\) can be resolved by the Koszul complex. Moreover, the restriction map \(H^0(Y, \mathcal{O}_Y(k)) \rightarrow H^0(X, \mathcal{O}_X(k))\) is surjective (see the proof of~\cite[Lemma~2.2]{konno/torelli}). At last, smooth complete intersections of general type in generalised Grassmannians satisfy the infinitesimal Torelli theorem by~\cite[Theorem~2.6]{konno/torelli}.

  Weighted projective spaces \(\mathbb{P}(\rho)\) are Mori dream spaces with \(\Cl(\mathbb{P}(\rho)) \simeq \mathbb{Z}\), and can be defined over \(\mathbb{Q}\). Lefschetz hyperplane section theorem and the adjunction formula hold for any smooth well-formed weighted complete intersection \(X \subset \mathbb{P}(\rho)\) (see Propositions~\ref{proposition:WCI-QS-WF-lefschetz} and~\ref{proposition:WCI-QS-WF-adjunction}). The restriction map \(H^0(\mathbb{P}(\rho), \mathcal{O}_{\mathbb{P}(\rho)}(k)) \rightarrow H^0(X, \mathcal{O}_X(k))\) is surjective (see Lemma~\ref{lemma:WCI-QS-WF-projective-normality}). Moreover, Proposition~\ref{proposition:WCI-smooth-WF-criterion} states that a weighted complete intersection is smooth and well-formed if and only if it is quasi-smooth and does not intersect the singular locus of \(\mathbb{P}(\rho)\). Consequently, a smooth divisor on a smooth well-formed weighted complete intersection \(X\) is again a smooth well-formed weighted complete intersection, and any conjugate \(X^{\sigma}\) is also smooth and well-formed for any \(\sigma \in \Aut(\mathbb{C} / \mathbb{Q})\) (cf. Lemma~\ref{lemma:CI-MDS-one-conjugation}). At last, any smooth well-formed weighted complete intersection of general type \(W \subseteq \mathbb{P}(\rho)\) such that \(\dim(\vert \mathcal{O}_W(1) \vert) \geqslant 2\) satisfies the infinitesimal Torelli theorem by Lemma~\ref{lemma:WCI-QS-WF-O1-dimension} and~\cite[Theorem~2.1]{usui/torelli}.

  Note that the linear system \(\vert \mathcal{O}_X(1) \vert\) is \emph{very} ample if and only if \(\mathbb{P}(\rho) = \mathbb{P}^n\). Consequently, if \(\mathbb{P}(\rho) \neq \mathbb{P}^n\), then a smooth divisor \(D \subset X\) lying in a very ample linear system satisfies \(D \not \in \vert \mathcal{O}_X(1) \vert\), hence \(\dim(\vert \mathcal{O}_X(1) \vert) = \dim(\vert \mathcal{O}_D(1) \vert)\) (see Lemma~\ref{lemma:WCI-QS-WF-O1-dimension}). In other words, we are able to apply Usui's infinitesimal Torelli theorem.
\end{proof}

\begin{question}\label{question:torelli}
  The only reason we assume \(\dim(\vert \mathcal{O}_X(1) \vert) \geqslant 2\) is because it is crucial for Usui's proof of the infinitesimal Torelli theorem (see~\cite{usui/torelli}, and also~\cite{licht/torelli}). Let us briefly recall its strategy. Let \(X \subset \mathbb{P}(a_0, \ldots, a_N)\) be a smooth well-formed weighted complete intersection of general type defined by a regular sequence \((f_1, \ldots, f_c)\) of weighted homogeneous polynomials of degree \(d_j\). We consider the \emph{Jacobi ring} of \(X\):
  \[
    R = \mathbb{C}[x_0, \ldots, x_N; y_1, \ldots, y_c]/(\partial_{x_0}(F), \ldots, \partial_{x_N}(F),
    \partial_{y_1}(F), \ldots, \partial_{y_c}(F)), \; F = \sum_{j = 1}^c y_j f_j.
  \]
  We denote by \(R_{(i, j)}\) its homogeneous components under the bi-grading \(\deg(x_i) = (0, 1)\), \(\deg(y_j) = (1, -d_j)\). Then the infinitesimal Torelli map
  \[
    \Psi \colon H^1(X, \mathcal{T}_X) \rightarrow \bigoplus_{p + q = \dim(X)} \Hom(H^p(X, \Omega_X^q), H^{p + 1}(X, \Omega_X^{q - 1}))
  \]
  can be interpreted as follows:
  \begin{gather*}
    H^1(X, \mathcal{T}_X) \simeq R_{(1, 0)}, \quad
    H^p(X, \Omega_X^q) \simeq \Hom(R_{(q, -S)}, \mathbb{C}), \quad
    S = \sum_{i = 0}^N a_i - \sum_{j = 1}^c d_j, \\
    \Psi_q \colon R_{(1, 0)} \rightarrow \Hom(\Hom(R_{(q, -S)}, \mathbb{C}), \Hom(R_{(q - 1, -S)}, \mathbb{C})) = \Hom(R_{(q - 1, -S)}, R_{(q, -S)}).
  \end{gather*}
  In other words, \(\Psi\) sends any \(r \in R_{1, 0}\) to the multiplication-by-\(r\) map, and the infinitesimal Torelli theorem holds precisely when \(\Psi\) is non-degenerated. Usui proved this under the assumption \(\dim(\vert \mathcal{O}_X(1) \vert) \geqslant 2\). Can the assumption be dropped?
\end{question}

\begin{definition}
  Let \(X\) be a smooth projective variety.

  A \emph{Lefschetz fibration} \(f \colon Y = \Bl_{\Bs(\mathcal{L})}(X) \rightarrow \mathbb{P}^{1}\) over \(X\) is a resolution of indeterminacy of the rational map \(\Phi_{\mathcal{L}} \colon X \dashrightarrow \mathbb{P}^1\) defined by a Lefschetz pencil~\(\mathcal{L}\) by a single blow up. A morphism \(f\) is proper and flat (see~\cite[Expos\'{e}~XVII]{deligne/sga}). A Lefschetz fibration \(f\) is \emph{rigid} if it is rigid as a flat family over~\(\mathbb{P}^1\). Equivalently, the corestriction \(\left.f\right|_{\mathbb{P}^{1}(\mathbb{C}) \setminus \Crit(f)}\) is a rigid smooth family (see~\cite[Proposition~III.9.8]{hartshorne/algebraic}).
\end{definition}

\begin{lemma}\label{lemma:ridig-lefschetz-fibrations-bounded}
  Let \(\varnothing \neq B \subseteq \mathbb{P}^1(\mathbb{C})\) be a Zariski-open subset, and \(h \in \mathbb{Z}[T]\) be an integer polynomial. Denote by \(\mathcal{X}_{B, h}\) the set of isomorphism classes over \(\mathbb{P}^1\) of all Lefschetz fibrations \(f \colon Y = \Bl_{\Bs(\mathcal{L})}(X) \rightarrow \mathbb{P}^1\) over all smooth projective varieties~\(X\) such that the corestriction \(\left.f\right|_{\mathbb{P}^{1}(\mathbb{C}) \setminus \Crit(f)}\) is a rigid smooth family of canonically polarised varieties with the Hilbert polynomial~\(h\), and the morphism \(f\) is smooth over~\(B\), i.e., we have \(B \subseteq \mathbb{P}^1(\mathbb{C}) \setminus \Crit(f)\). Then the set \(\mathcal{X}_{B, h}\) is finite.
\end{lemma}

\begin{proof}
  Follows from the result of Kov\'{a}cs--Lieblich (see~\cite[Theorem~2.2]{kovacs/boundedness}).
\end{proof}

\begin{corollary}\label{corollary:lefschetz-fibrations-finiteness}
  Let \(\varnothing \neq B \subseteq \mathbb{P}^1(\mathbb{C})\) be a Zariski-open subset, and \(h \in \mathbb{Z}[T]\) be an integer polynomial. Denote by \(\mathcal{Y}_{B, h, n}\) the set of isomorphism classes of smooth projective varieties of general type \(X \subseteq \mathbb{P}^{n}\) admitting a Lefschetz fibration \(f \colon Y \rightarrow \mathbb{P}^{1}\) over \(X\) such that the corestriction \(\left.f\right|_{\mathbb{P}^{1}(\mathbb{C}) \setminus \Crit(f)}\) is a rigid smooth family of canonically polarised varieties with the Hilbert polynomial \(h\), and the morphism \(f\) is smooth over~\(B\), i.e., we have \(B \subseteq \mathbb{P}^1(\mathbb{C}) \setminus \Crit(f)\). Then the set \(\mathcal{Y}_{B, h, n}\) is finite.
\end{corollary}

\begin{proof}
  Follows from Lemma~\ref{lemma:ridig-lefschetz-fibrations-bounded} and a Tsai-type theorem for bounded families (see~\cite[Theorem~2.2]{javanpeykar/belyi}).
\end{proof}

\begin{lemma}\label{lemma:lefschetz-fibration-rigidity}
  Let \(X\) be a smooth projective variety, and \(f \colon Y \rightarrow \mathbb{P}^1\) be a Lefschetz fibration over~\(X\). Assume that fibres of the corestriction \(\left.f\right|_B\) to \(B = \mathbb{P}^1(\mathbb{C}) \setminus \Crit(f)\) satisfy the infinitesimal Torelli theorem.

  Then \(\left.f\right|_B\) is a rigid smooth family (i.e., \(f \colon Y \rightarrow \mathbb{P}^1\) is a rigid Lefschetz fibration).
\end{lemma}

\begin{proof}
  Follows verbatim the proof of~\cite[Theorem~2.5]{javanpeykar/belyi}.
\end{proof}

\begin{lemma}\label{lemma:conjugation-finiteness}
  Let \(X\) be a smooth complex projective variety, and \(f \colon Y \rightarrow \mathbb{P}^1\) be a Lefschetz fibration over~\(X\). Put \(B = \mathbb{P}^1(\mathbb{C}) \setminus \Crit(f)\). Assume that \(\Crit(f) \subseteq \mathbb{P}^1(\overline{\mathbb{Q}})\), and the orbit of the morphism \(f^{\sigma} \colon Y^{\sigma} \rightarrow \mathbb{P}^1\) under the action of \(\Stab_{\Aut(\mathbb{C} / \mathbb{Q})}(B^{\sigma})\) is finite for any \(\sigma \in \Aut(\mathbb{C} / \mathbb{Q})\). Then the morphism \(f \colon Y \rightarrow \mathbb{P}^1\) can be defined over \(\overline{\mathbb{Q}}\).
\end{lemma}

\begin{proof}
  Gonz\'{a}lez-Diez's weak version of Weil's criterion of rationality (see~\cite[60--61]{gonzalez-diez/belyi}) states that our statement is equivalent to the following one: the set of conjugates \(\{f^{\sigma} \colon Y^{\sigma} \rightarrow \mathbb{P}^1\}\) under the \(\Aut(\mathbb{C} / \mathbb{Q})\)-action is finite. In turn, this follows from our assumptions and~\cite[Proposition~III.9.8]{hartshorne/algebraic}.
\end{proof}

\begin{theorem}[see~{\cite[Theorem~4.2]{javanpeykar/belyi}}]\label{theorem:CI-admissible}
  Let \(Y\) be a complex Mori dream space with \(\Cl(Y) \simeq \mathbb{Z}\) which can be defined over~\(\mathbb{Q}\), and \(X \subseteq Y\) be a smooth complex admissible complete intersection (see Definition~\ref{definition:CI-admissible}) of general type of dimension at least 3. The following assertions are equivalent.

  \begin{enumerate}
  \item The variety \(X\) can be defined over \(\overline{\mathbb{Q}}\).
  \item There exists a Lefschetz pencil \(f \colon X \dashrightarrow \mathbb{P}^1\) with \(\Crit(f) \subset \mathbb{P}^1(\overline{\mathbb{Q}})\).
  \item There exists a Lefschetz function \(f \colon X \dashrightarrow \mathbb{P}^1\) with \(\Crit(f) \subset \mathbb{P}^1(\overline{\mathbb{Q}})\).
  \item There exists a Lefschetz function \(f \colon X \dashrightarrow \mathbb{P}^1\) with \(|\Crit(f)| \leqslant 3\).
  \item There exists a Lefschetz fibration \(\widetilde{f} \colon Y \rightarrow \mathbb{P}^1\) over \(X\) with \(\Crit(\widetilde{f}) \subset \mathbb{P}^1(\overline{\mathbb{Q}})\).
  \end{enumerate}
\end{theorem}

\begin{proof}
  The implication \((1)\Longrightarrow(2)\) follows from the existence of a Lefschetz pencil on a projective variety over \(\overline{\mathbb{Q}}\). The implications \((2)\Longrightarrow(3)\) and \((4)\Longrightarrow(5)\) hold by definition. The implication \((3)\Longrightarrow(4)\) follows from Belyi's algorithm (see~\cite{belyi/extensions}). Let us prove the implication \((5)\Longrightarrow(1)\).

  Assume that there exists a Lefschetz fibration \(f \colon Y \rightarrow \mathbb{P}^1\) over \(X\) whose critical points lie in \(\mathbb{P}^1(\overline{\mathbb{Q}})\). Put \(B = \mathbb{P}^1(\mathbb{C}) \setminus \Crit(f)\). By definition a Lefschetz pencil is contained in  a very ample linear system. We have assumed \(X \subset Y\) to be an admissible complete intersection, hence the corestriction \(\left.f\right|_B\) is a smooth family of canonically polarised admissible complete intersections in \(Y\). Actually, fibres of \(\left.f\right|_B\) are admissible complete intersections as well, and we only have to apply the adjunction formula. Moreover, by the same assumption fibres of \(\left.f\right|_B\) satisfy the infinitesimal Torelli theorem. Then \(\left.f\right|_B\) is a rigid family by Lemma~\ref{lemma:lefschetz-fibration-rigidity}.

  Recall that \(\Aut(\mathbb{C} / \mathbb{Q})\)-action preserves the multidegree of a complete intersection by Lemma~\ref{lemma:CI-MDS-one-conjugation}. Consequently, the Hilbert polynomial of \(X^{\sigma}\) also does not depend on \(\sigma \in \Aut(\mathbb{C} / \mathbb{Q})\), since all of them are deformationally equivalent. Then we can apply Corollary~\ref{corollary:lefschetz-fibrations-finiteness} to conclude that the number of conjugates of \(f\) under the action of \(\Stab_{\Aut(\mathbb{C} / \mathbb{Q})}(B^{\sigma})\) is finite for any \(\sigma \in \Aut(\mathbb{C} / \mathbb{Q})\). So the morphism \(f \colon Y \rightarrow \mathbb{P}^1\) and the variety \(Y\) can be defined over~\(\overline{\mathbb{Q}}\) by Lemma~\ref{lemma:conjugation-finiteness}. At last, this implies that \(X\) itself can be defined over \(\overline{\mathbb{Q}}\) (see~\cite[Lemma~3.3]{javanpeykar/belyi}).
\end{proof}

\begin{proof}[Proof of Theorems~\ref{theorem:CI-grassmannians} and~\ref{theorem:WCI}]
  Let us assume that one of the following holds:
  \begin{itemize}
  \item \(X \subset Y\) is a smooth complex complete intersection of general type of dimension at least~3 in a generalised Grassmannian \(Y\);
  \item \(X \subset \mathbb{P}(\rho)\) be a smooth well-formed complex weighted complete intersection of general type of dimension at least 3, and \(\dim(\vert \mathcal{O}_X(1) \vert) \geqslant 2\).
  \end{itemize}
  Proposition~\ref{proposition:CI-admissible} implies that \(X\) is an admissible complete intersection (in the sense of Definition~\ref{definition:CI-admissible}) of general type. Then we can apply Theorem~\ref{theorem:CI-admissible}.
\end{proof}

\section{A possible generalisation}\label{section:generalisation}

As we have seen in the previous section, Belyi's theorem holds for smooth (well-formed) complete intersections of general type in a generalised Grassmannian or a weighted projective space. Moreover, we have proved Belyi's theorem for any \emph{admissible} smooth complete intersection of general type in a complex Mori dream space \(Y\) with \(\Cl(Y) \simeq \mathbb{Z}\) which can be defined over \(\mathbb{Q}\). It is natural to ask whether this setting generalises to \emph{weighted generalised Grassmannians}.

\begin{definition}[see~\cites{corti/weighted,qureshi/grassmannians}]
  A Mori dream space \(Y\) with \(\Cl(Y) \simeq \mathbb{Z}\) is a \emph{weighted generalised Grassmannian} if its Cox ring \(\mathcal{R}(Y)\) admits homogeneous generators such that the associated ideal of relations coincides with the ideal of relations of the Cox ring of a usual generalised Grassmannian.

  In other words, a weighted generalised Grassmannian is a closed subvariety of a weighted projective space defined by the relations of a usual generalised Grassmannian in its Cox ring (see~\cite[\nopp 3.2.3]{arzhantsev/cox} for explicit description).
\end{definition}

So \(Y\) is a Mori dream space with \(\Cl(Y) \simeq \mathbb{Z}\), and can be defined over~\(\mathbb{Q}\). Weighted Grassmannians in the sense of~\cite{corti/weighted} is a special case of weighted generalised Grassmannians. We refer the reader to~\cite{qureshi/flag-I,qureshi/threefolds,qureshi/flag-II} for explicit examples.

Weighted projective spaces are a special case of weighted generalised Grassmannians, hence we have to introduce the notion of well-formedness (cf. Subsection~\ref{subsection:WCI}).

\begin{definition}[cf. Definition~\ref{definition:WF-subscheme}]
  Let \(Y \subset \mathbb{P}(\rho)\) be a weighted generalised Grassmannian. It is \emph{well-formed} if \(\codim_Y(Y \cap \Sing(\mathbb{P}(\rho))) \geqslant 2\). A closed subscheme \(X \subseteq Y\) is \emph{well-formed} if \(Y \subset \mathbb{P}(\rho)\) is well-formed, and \(\codim_X(X \cap \Sing(Y)) \geqslant 2\).
\end{definition}

Theorems~\ref{theorem:CI-grassmannians} and~\ref{theorem:WCI} lead us to the following natural question.

\begin{question}\label{question:main}
  Are smooth well-formed complete intersections in weighted generalised Grassmannians admissible in the sense of Definition~\ref{definition:CI-admissible}?
\end{question}

\begin{corollary}[see Theorem~\ref{theorem:CI-admissible}]
  Assume that Question~\ref{question:main} has a positive answer. Let \(Y\) be a weighted generalised Grassmannian, and \(X \subset Y\) be a smooth well-formed complex complete intersection of general type of dimension at least 3. Then the variety \(X\) can be defined over \(\overline{\mathbb{Q}}\) if and only if there exists a Lefschetz function \(X \dashrightarrow \mathbb{P}^1\) with at most 3 critical points.
\end{corollary}

\begin{example}
  Let \(V = \langle e_0, e_1, e_2, e_3 \rangle\) and \(W = \wedge^2(V) = \langle e_{0, 1}, \ldots, e_{2, 3} \rangle\) be vector spaces, and \((a_0, a_1, a_2, a_3)\) be a tuple of integers such that \(a_i + a_j > a_0\) for any \(0 \leqslant i < j \leqslant 3\). If we denote by \(T_{i, j}\) the coordinates on \(W\), then the weighted projective space \(\mathbb{P} = \mathbb{P}(a_{0, 1}, a_{0,2}, a_{0,3}, a_{1,2}, a_{1,3}, a_{2, 3})\) with the following weights:
  \[
    a_{i, j} = \deg(T_{i, j}) = - a_0 + a_i + a_j > 0
  \]
  contains a weighted hypersurface \(Y \subset \mathbb{P}\) given by the Pl\"{u}cker relation
  \[
    T_{0, 1} T_{2, 3} + T_{1, 2} T_{0, 3} - T_{0, 2} T_{1, 3} = 0,
  \]
  which is quasi-homogeneous of the degree \(-a_0 + a_1 + a_2 + a_3\). Let us assume that \(Y\) is well-formed, then the adjunction formula implies that
  \[
    \deg(-K_{\mathbb{P}}) = 3 (-a_0 + a_1 + a_2 + a_3), \quad
    \deg(-K_Y) = 2 (-a_0 + a_1 + a_2 + a_3).
  \]

  Now let \(\widetilde{X} \subset \mathbb{P}\) be a general hypersurface of degree \(d\) such that \(X = Y \cap \widetilde{X} \subset \mathbb{P}\) is a smooth well-formed complete intersection of the multidegree \((d, -a_0 + a_1 + a_2 + a_3)\). We can think of \(X\) as a smooth well-formed hypersurface \(X \subset Y\) of degree \(d\) whose degree of the anticanonical class is equal to
  \[
    \deg(-K_X) = \deg(-K_{\mathbb{P}}) - (-a_0 + a_1 + a_2 + a_3) - d = \deg(-K_Y) - d.
  \]
  If \(d > -a_0 + a_1 + a_2 + a_3\), then \(X\) is of general type, so we can apply Theorem~\ref{theorem:WCI}.
\end{example}

\clearpage
\printbibliography

\end{document}